\documentclass{amsart}
\usepackage{amsthm,amsmath,amsfonts,amssymb,array,enumerate,verbatim,enumitem,stmaryrd}
\usepackage[all]{xy}

\newcommand{\isomto}{\overset{\sim}{\rightarrow}}

\newcommand{\CC}{\mathbb{C}}
\newcommand{\FF}{\mathbb{F}}
\newcommand{\ZZ}{\mathbb{Z}}

\newcommand{\QQ}{\mathbb{Q}}

\newcommand{\cO}{\mathcal{O}}

\newcommand{\fm}{\mathfrak{m}}

\newcommand{\fW}{\mathfrak{W}}

\newcommand{\cW}{\mathcal{W}}
\newcommand{\cK}{\mathcal{K}}

\newcommand{\cZ}{\mathcal{Z}}

\newcommand{\zlb}{\overline{\mathbb{Z}_{\ell}}}
\newcommand{\flb}{\overline{\mathbb{F}_{\ell}}}

\newcommand{\cKb}{\overline{\cK}}

\DeclareMathOperator{\Ind}{Ind}
\DeclareMathOperator{\cInd}{c-Ind}
\DeclareMathOperator{\End}{End}
\DeclareMathOperator{\Rep}{Rep}

\DeclareMathOperator{\Frac}{Frac}

\DeclareMathOperator{\nil}{nil}
\DeclareMathOperator{\scs}{scs}
\DeclareMathOperator{\gen}{gen}

\newtheorem{thm}{Theorem}[section]
\newtheorem{lemma}[thm]{Lemma}
\newtheorem{prop}[thm]{Proposition}
\newtheorem{cor}[thm]{Corollary}
\newtheorem{defn}[thm]{Definition}
\newtheorem{ex}[thm]{Example}

\newtheorem{rmk}[thm]{Remark}

\newtheorem{question}[thm]{Question}

\newtheorem*{thm*}{Theorem}

\begin{document}

\title[Characterizing mod-$\ell$ local Langlands]{Characterizing the mod-$\ell$ local Langlands correspondence by nilpotent gamma factors}
\author{Gilbert Moss}
\date{\today}
\maketitle
\begin{abstract}
Let $F$ be a $p$-adic field and choose $k$ an algebraic closure of $\mathbb{F}_{\ell}$, with $\ell$ different from $p$. We define ``nilpotent lifts'' of irreducible generic $k$-representations of $GL_n(F)$, which take coefficients in Artin local $k$-algebras. We show that an irreducible generic $\ell$-modular representation $\pi$ of $GL_n(F)$ is uniquely determined by its collection of Rankin--Selberg gamma factors $\gamma(\pi\times \widetilde{\tau},X,\psi)$ as $\widetilde{\tau}$ varies over nilpotent lifts of irreducible generic $k$-representations $\tau$ of $GL_t(F)$ for $t=1,\dots, \lfloor \frac{n}{2}\rfloor$. This gives a characterization of the mod-$\ell$ local Langlands correspondence in terms of gamma factors, assuming it can be extended to a surjective local Langlands correspondence on nilpotent lifts.
\end{abstract}

\section{Introduction}

\subsection{Notation}
Let $F$ be a finite extension of $\QQ_p$ with finite residue field of order $q$, let $G:=G_n:=GL_n(F)$, let $R$ be a commutative ring with unit, and let $\Rep_R(G_n)$ be the category of smooth $R[G_n]$-modules (the stabilizer of any element is open). Let $\ell$ be a prime different from $p$, let $k:=\flb$, and let $W(k)$ be the ring of Witt vectors (it is the $\ell$-adic completion of the ring of integers in the maximal unramified extension of $\QQ_{\ell}$). Let $\cK = \Frac(W(k))= W(k)[\frac{1}{\ell}]$, let $\overline{\cK}$ be an algebraic closure, and let $\cO$ be the ring of integers in $\overline{\cK}$. After fixing a field isomorphism $\cKb\cong \CC$ we may translate the properties of $\Rep_{\CC}(G_n)$ to $\Rep_{\cKb}(G_n)$.

Fix a nontrival additive character $\psi:F\to W(k)^{\times}$. For any $W(k)$-algebra $W(k)\to R$ we denote by $\psi_R$ the extension $F\to W(k)^{\times}\to R^{\times}$. Extend $\psi_R$ to the subgroup $U$ of unipotent upper triangular matrices by $\psi(u)=\psi(u_{1,2}+\cdots+u_{n-1,n})$. Given a $W(k)$-algebra $R$, we say $\pi\in \Rep_R(G_n)$ is \emph{generic} if there exists a nontrivial homomorphism 
$$\pi\to \Ind_U^G\psi_R:=\left\{W:G\to R\text{ smooth}:W(ug)=\psi_R(u)W(g),\ u\in U,\  g\in G\right\}.$$ If $R$ is an algebraically closed field of characteristic different from $p$, $\mathcal{A}_{R}^{\gen}(n)$ will denote the set of isomorphism classes of irreducible generic objects in $\Rep_R(G)$. Given $\pi\in \mathcal{A}_R^{\gen}(n)$, the map $\pi\to \Ind_U^G\psi_R$ is unique up to scaling (\cite[III.1.11]{vig}); its image is denoted $\cW(\pi,\psi)$ and called the \emph{Whittaker model} of $\pi$, and $\pi\cong\cW(\pi,\psi)$ by irreducibility.

\subsection{On the mod-$\ell$ converse theorem}
Given integers $n$, $t\geq 1$, and $\pi\in \mathcal{A}_{\cKb}^{\gen}(n)$, $\tau\in \mathcal{A}_{\cKb}^{\gen}(t)$, Jacquet, Piatetski-Shapiro, and Shalika in \cite{jps1} construct gamma factors of pairs $\gamma(\pi\times\tau,X,\psi)\in\cKb(X)$ satisfying a certain functional equation. If $\pi_1$, $\pi_2$ are isomorphic, then $\gamma(\pi_1\times\tau,X,\psi)=\gamma(\pi_2\times \tau,X,\psi)$ for all $\tau\in \mathcal{A}_{\cKb}^{\gen}(t)$, for all $t\geq 1$. A ``local converse theorem'' identifies a collection of representations $\tau$ such that the converse statement holds, i.e. such that the collection of $\gamma(\pi\times\tau,X,\psi)$ uniquely determines $\pi\in \mathcal{A}_{\cKb}^{\gen}(n)$. The first converse theorems appear in \cite{jacquet_langlands, jps2} for $n=2$ and $3$, respectively; in both cases $\tau$ runs over characters of $G_1$. In \cite{hen_converse} Henniart proves a converse theorem for arbitrary $n$, where $\tau$ runs over irreducible generic objects in $\Rep_{\cKb}(G_t)$ for $t=1,2,\dots,n-1$. 

In \cite{hen_converse}, Henniart applies the converse theorem to prove that the local Langlands correspondence over $\cKb$ is uniquely characterized by the property that it equates the Deligne--Langlands gamma factors (\cite{deligne72}) with the Jacquet--Piatetski-Shapiro--Shalika gamma factors (\cite{jps2}). Vign\'{e}ras proves in \cite{vig_ll} the existence of a local Langlands correspondence over $k$. It is uniquely characterized not by gamma factors but by its compatibility with the $\cKb$-correspondence under reduction mod-$\ell$ (see \S~\ref{characterizinglocallanglands} below for more details). 

The goal of this article is to prove a mod-$\ell$ converse theorem and use it to uniquely characterize the mod-$\ell$ local Langlands correspondence intrinsically by mod-$\ell$ gamma factors, without referring to the $\cKb$-setting. We will now describe our mod-$\ell$ converse theorem, and postpone until \S~\ref{characterizinglocallanglands} the question of characterizing the mod-$\ell$ local Langlands correspondence.

The construction of gamma factors associated to objects in $\mathcal{A}_R^{\gen}(n)$ is already established in the literature when $R=k=\flb$. Gamma factors of pairs are developed over $\flb$ by Kurinczuk--Matringe in \cite{km}, and shown to be compatible with the reduction mod-$\ell$ of the gamma factors of \cite{jps2}. At first, one might hope for a mod-$\ell$ converse theorem that parallels the situation over $\cKb$: that the gamma factors $\gamma(\pi\times\tau,X,\psi)$ uniquely characterize $\pi\in\mathcal{A}_{k}^{\gen}(n)$ as $\tau$ ranges over $\mathcal{A}^{\gen}_k(t)$ for $t=1,\dots,n-1$. However, we give a counterexample in Section~\ref{counterexample} when $n=2$ showing this hope is false. Thus a new framework is needed to achieve a mod-$\ell$ converse theorem. 

When examining the counterexample of Section~\ref{counterexample}, one finds that the lack of nontrivial $\ell$-power roots of unity in $k$ leads to a lack of tamely ramified $k$-valued characters. This leads to an excess of congruences mod $\ell$ between gamma factors $\gamma(\pi_{\theta}\times \chi,X,\psi)$ as $\chi$ varies, where $\pi_{\theta}$ is a certain integral object in $\mathcal{A}_{\overline{\cK}}^{\gen}(2)$ constructed in \S~\ref{counterexample}. The same problem arises when trying to adapt the traditional method of proving converse theorems over $\mathbb{C}$, where a crucial ingredient is the completeness of Whittaker models with respect to the $L^2$-inner product. We give an example in Section~\ref{vanishinglemmasgonewrong} showing that completeness of Whittaker models fails over $k$, again due to the lack of nontrivial $\ell$-power roots of unity in $k$. Their appearence in the mod-$\ell$ Bernstein center (c.f. Example~\ref{lpowerrootsbernstein}) is another example of the importance of $\ell$-power roots of unity to the mod-$\ell$ representation theory.

To recover the converse theorem, we pass from $k$ to a larger class of $k$-algebras that has an ample supply of $\ell$-power roots of unity: Artin local $k$-algebras, or equivalently, finite-dimensional local $k$-algebras. Such rings $R$ have only a single prime ideal (namely, their nilradical) they have residue field $k$, and the composition $k\to R\to k$ is the identity map.

\begin{defn}\label{nilpotentliftdef}
Given an object $\tau\in\mathcal{A}_{k}^{\gen}(t)$, a \emph{nilpotent lift} of $\tau$ to an Artin local $k$-algebra $R$ is an admissible $R[G_t]$-submodule $\widetilde{\tau}$ of $\Ind_U^G\psi_{R}$ for which
$$\widetilde{\tau}\otimes_{R}k=\cW(\tau,\psi).$$

For an Artin local $k$-algebra $R$, let $\mathcal{A}^{\gen}_R(n)$ denote the set of isomorphism classes of nilpotent lifts to $R$ of objects in $\mathcal{A}^{\gen}_{k}(t)$. 

Let $\mathcal{A}^{\gen}_{\nil}(t)$ denote the set of isomorphism classes of all nilpotent lifts of objects in $\mathcal{A}^{\gen}_{k}(t)$, to any $R$.
\end{defn}
\noindent In other words, a nilpotent lift is an infinitesimal deformation of the Whittaker model of $\tau$.

Given $\tau\in \mathcal{A}_k^{\gen}(t)$, the isomorphism $\tau\cong \cW(\tau,\psi)$ allows us to identify $\tau$ with a nilpotent lift of itself, so we have an ``inclusion'' $\mathcal{A}_{k}^{\gen}(t)\subset\mathcal{A}_{\nil}^{\gen}(t)$. An extension of the theory of gamma factors that encompasses nilpotent lifts is carried out in the author's thesis (\cite{moss1,moss2}). We refer to \S~\ref{section:cowhittaker} and \S~\ref{section:gamma} for a summary of these gamma factors. Our main result is the following converse theorem.

\begin{thm}\label{conversemain}
Let $\pi_1$, $\pi_2$ be in $\mathcal{A}_{k}^{\gen}(n)$, $n\geq 2$. Suppose
\begin{equation}\label{equalityofnilpotentgammas}
\gamma(\pi_1\times\widetilde{\tau},X,\psi) = \gamma(\pi_2\times\widetilde{\tau},X,\psi)
\end{equation}
for all nilpotent lifts $\widetilde{\tau}\in \mathcal{A}_{\nil}^{\gen}(t)$, for all $t=1,2,\dots,\lfloor \frac{n}{2}\rfloor$. Then $\pi_1\cong \pi_2$.
\end{thm}

In fact, the theorem is still true if $\pi_1$ and $\pi_2$ are in $\mathcal{A}_{\nil}^{\gen}(n)$ (see Theorem~\ref{conversefull}).

Theorem~\ref{conversemain} addresses a conjecture of Vign\'{e}ras from 2000. In \cite{vig_epsilon}, Vign\'{e}ras defines modified gamma factors for cuspidal objects in $\mathcal{A}^{\gen}_{k}(2)$ and proves they satisfy a mod-$\ell$ converse theorem. In the same article, she conjectures that an $\ell$-modular converse theorem holds for all $n>2$, using some suitably modified version of the gamma factor. This is discussed in more detail in Section~\ref{otherwork}.

When $n>2$ there is another aspect of the converse theorem, which is obtaining the \emph{minimal} range of $t$ such that the collection $\gamma(\pi\times\tau,X,\psi)$ determines $\pi$ as $\tau$ varies. For example, over $\cKb$, Henniart's proof \cite{hen_converse} requires $\mathcal{A}_{\cKb}(t)$ for all $t$ within the range $1\leq t \leq n-1$. When $n=3$, it is shown in \cite{jps2} that $t=1$ suffices, and Jacquet conjectured in 1999 that $1\leq t\leq \lfloor\frac{n}{2}\rfloor$ should suffice for arbitrary $n$. Jacquet's conjecture was proven in \cite{chai_converse} and \cite{jl}, independently, for all $n$. We thus present Theorem~\ref{conversemain} as a mod-$\ell$ analogue of Jacquet's conjecture. It is shown in \cite{sharpness} that the bound $t\leq \lfloor \frac{n}{2}\rfloor$ cannot be improved in the $\overline{\cK}$-setting.

\subsection{Characterizing the mod-$\ell$ Langlands correspondence}\label{characterizinglocallanglands}

Suppose for the time being that $R$ is any algebraically closed field of characteristic $\ell$ different from $p$. An irreducible representation in $\Rep_R(G_n)$ is \emph{supercuspidal} (respectively, \emph{cuspidal}) if it is not isomorphic to a subquotient (respectively, quotient) of any parabolic induction from a proper Levi subgroup (the two notions are equivalent when the characteristic of $R$ is zero). Let $\mathcal{A}^0_R(n)$ denote the set of isomorphism classes of irreducible supercuspidal representations in $\Rep_{R}(G_n)$. Given $\pi\in \Rep_R(G_n)$ irreducible, the \emph{supercuspidal support} of $\pi$, denoted $\scs(\pi)$, is the unique multiset $\{\pi_1,\dots,\pi_k\}$ such that $\pi$ is a subquotient of the normalized parabolic induction $\pi_1\times\cdots\times\pi_k$.  If $\mathcal{A}_R^{\gen}(n)$ denotes the set of irreducible generic representations in $\Rep_R(G_n)$, and $\mathcal{A}_R^{\scs}(n)$ denotes the set of supercuspidal supports of irreducible representations in $\Rep_{R}(G_n)$, the map $\tau\mapsto \scs(\tau)$ defines a bijection $\mathcal{A}_R^{\gen}(n)\to \mathcal{A}_R^{\scs}(n)$ (\cite[III.1.11]{vig}). 

Choose an algebraic closure of $F$, let $W_F$ be the absolute Weil group, let $\mathcal{G}_R^0(n)$ denote the set of isomorphism classes of $n$-dimensional irreducible smooth $R$-representations of $W_F$. A local Langlands correspondence over $R$ starts with a sequence of bijections $L_{R,n}^0:\mathcal{G}_{R}^0(n)\to \mathcal{A}_{R}^0(n)$, $n=1,2,\dots$. If $\mathcal{G}_R^{ss}(n)$ denotes the set of isomorphism classes of semisimple $n$-dimensional $R$-representations of $W_F$, a \emph{semisimple} local Langlands correspondence over $R$ is a sequence $(L_{R,n}^{ss})_{n\geq 1}$ of bijections $L_{R,n}^{ss}:\mathcal{G}_{R}^{ss}(n)\to\mathcal{A}_{R}^{\scs}(n)$ extending a sequence $L_{R,n}^0:\mathcal{G}_R^0(n)\to \mathcal{A}_R^0(n)$ according to the relation
$$L_{R,n}^{ss}(\rho_1\oplus \cdots \oplus\rho_k) = \{L_{R,n}^0(\rho_1),\dots,L_{R,n}^0(\rho_k)\}.$$
Composing with the inverse of the bijection $\mathcal{A}_R^{\gen}(n)\isomto \mathcal{A}_R^{\scs}(n)$ gives a \emph{generic} local Langlands correspondence 
$$L^{\gen}_{R,n}:\mathcal{G}^{ss}_R(n)\to \mathcal{A}_R^{\gen}(n).$$

In \cite{harris_taylor} (and \cite{lrs} for positive characteristic $F$), it is shown for $R=\cKb$ that there exists a sequence $L_{\cKb,n}^0:\mathcal{G}_{\cKb}^0(n)\to \mathcal{A}_{\cKb}^0(n)$ satisfying:
\begin{enumerate}[label=(\roman*)]
\item $L_{\cKb,1}^0$ is given by local class field theory,
\item $L_{\cKb,n}^0(\rho^{\vee})=L_{\cKb,n}^0(\rho)^{\vee}$, and 
\item for all pairs of integers $n,t\geq 1$, $\rho\in\mathcal{G}_{\cKb}^0(n)$, $\rho'\in \mathcal{G}_{\cKb}^0(t)$, $$\gamma(L_{\cKb,n}^0(\rho)\times L_{\cKb,t}^0(\rho),s,\psi)=\gamma(\rho\otimes\rho',s,\psi).$$ 
\end{enumerate}
where $\gamma(\rho\otimes\rho',s,\psi)$ denotes the Deligne--Langlands local factor of \cite{deligne72}. Since the Deligne--Langlands gamma factor is multiplicative in direct sums, and the Jacquet--Piatetski-Shapiro--Shalika gamma factor is multiplicative in normalized parabolic inductions, the generic Langlands correspondence $L_{\cKb,n}^{\gen}:\mathcal{G}_{\cKb}^{ss}(n)\to \mathcal{A}_{\cKb}^{\gen}(n)$ induced by $(L_{\cKb,n}^0)_{n\geq 1}$ satisfies:
\begin{enumerate}[label=(\roman*)]
\item $L_{\cKb,1}^{\gen}$ is given by local class field theory on $\mathcal{G}_{\cKb}^0(1)$ ($=\mathcal{G}_{\cKb}^{ss}(1)$),
\item for all pairs $n>t\geq 1$, $\rho\in\mathcal{G}_F^{ss}(n)$, $\rho'\in \mathcal{G}_{\cKb}^{ss}(t)$, $$\gamma(L_{\cKb,n}^{\gen}(\rho)\times L_{\cKb,t}^{\gen}(\rho'),s,\psi)=\gamma(\rho\otimes\rho',s,\psi),$$ 
\end{enumerate}
It then follows immediately from (ii) and the converse theorem of \cite{hen_converse} that there can be at most one sequence $(L_{\cKb,n}^{\gen})_{n\geq 1}$ with these properties.

For $\ell>0$ a prime, $\ell\neq p$, Vign\'{e}ras proves in \cite{vig_ll} the existence of a sequence of bijections $L_{k,n}^0:\mathcal{G}_{k}^0(n) \to \mathcal{A}_{k}^0(n)$, for all $n\geq 1$, which is uniquely characterized by the property that its induced semisimple correspondence $(L_{k,n}^{ss})_{n\geq 1}$ is compatible with $(L_{\cKb,n}^0)_{n\geq 1}$ under reduction mod-$\ell$. More precisely, a finite length object $\pi$ in $\Rep_{\cKb}(G_n)$ (or $\Rep_{\cKb}(W_F)$) is \emph{integral} if it admits a model $L$ over the ring of integers $\cO_E$ in a finite extension $E/\mathbb{Q}_{\ell}$ that is free over $\cO_E$ and finite over $\cO_E[G]$; write $r_{\ell}(\pi)$ to denote the semisimplified mod-$\ell$ reduction $(L\otimes_{\cO_E}k)^{ss}$. The reduction $r_{\ell}(\pi)$ of $\pi\in\mathcal{A}_{\cKb}^0(G_n)$ remains irreducible and cuspidal but may no longer be \emph{super}cuspidal, and the reduction $r_{\ell}(\rho)$ of $\rho\in\mathcal{G}_{\cKb}^0(n)$ may no longer be irreducible. Compatibility with $(L_{\cKb,n}^0)_{n\geq 1}$ under reduction mod-$\ell$ means that, for $\rho,\rho'\in\mathcal{G}^0_{\cKb}$ integral, \begin{itemize}
\item $L_{\cKb,n}^0(\rho)$ is integral
\item $r_{\ell}(L_{\cKb,n}^0(\rho)) =r_{\ell}(L_{\cKb,n}^0(\rho'))\iff r_{\ell}(\rho) = r_{\ell}(\rho')$, and
\item $L_{k,n}^{ss}(r_{\ell}(\rho)) =  \scs(r_{\ell}(L_{\cKb,n}^0(\rho)))$.
\end{itemize}
The uniqueness of the mod-$\ell$ correspondence $(L_{k,n}^{ss})_{n\geq 1}$ (and hence $(L_{k,n}^{\gen})_{n\geq 1}$) then follows from the uniqueness of $(L_{\cKb,n}^{ss})_{n\geq 1}$. 

Our goal is to use Theorem~\ref{conversemain} to characterize $(L_{k,n}^{\gen})_{n\geq 1}$ directly, in analogy with the characterization of $(L_{\cKb,n}^{\gen})_{n\geq 1}$. To accomplish this, we define nilpotent lifts of objects in $\mathcal{G}_{k}^{ss}$: 
\begin{defn}
Given $\rho\in \mathcal{G}_{k}^{ss}(n)$, a \emph{nilpotent lift} of $\rho$ to an Artin local $k$-algebra $R$ is a smooth $R[W_F]$-module $\widetilde{\rho}$, free of rank $n$ over $R$, such that there exists an isomorphism $$\widetilde{\rho}\otimes_Rk\cong \rho.$$ 

For an Artin local $k$-algebra $R$, let $\mathcal{G}_R(n)$ denote the set of isomorphism classes of nilpotent lifts to $R$ of objects in $\mathcal{G}_{k}^{ss}(n)$. 

Let $\mathcal{G}_{\nil}(n)$ denote the set of all isomorphism classes of nilpotent lifts of objects in $\mathcal{G}_k^{ss}(n)$, to any $R$. 
\end{defn}
\noindent In other words, a nilpotent lift is an infinitesimal deformation. We may identify $\mathcal{G}_{k}^{ss}(n)$ with a subset of $\mathcal{G}_{\nil}(n)$, and of $\mathcal{G}_R(n)$ for any $R$.

The theory of gamma factors of objects in $\mathcal{G}_{k}^{ss}(n)$ has long been established by work of Deligne \cite{deligne72}, but the classical theory does not accommodate nilpotent lifts. Fortunately, by previous work of the author and Helm (\cite{galois_gamma}), the Deligne--Langlands gamma factor generalizes to arbitrary Noetherian $W(k)$-algebras (see Theorem~\ref{thm:local galois gamma} for a summary). 

Our second main result characterizes the sequence $(L_{k,n}^{\gen})_{n\geq 1}$ by gamma factors, so far as it can be extended to a correspondence on nilpotent lifts.

\begin{thm}\label{characterization}
There exists \textbf{at most one} sequence of maps $$L_{k,n}^{\gen}:\mathcal{G}_{k}^{ss}(n)\to \mathcal{A}^{\gen}_{k}(n),\ \ n\geq 1$$ that admits, for every Artin local $k$-algebra $R$, an extension to a sequence of surjections $$L_{R,n}:\mathcal{G}_{R}(n)\to \mathcal{A}_{R}^{\gen}(n),\ \ n\geq 1$$ satisfying
\begin{enumerate}
\item $L_{R,1}$ is given by local class field theory,
\item For all $n> t$, $\rho\in \mathcal{G}_{R}(n)$, $\rho'\in\mathcal{G}_{R'}(t)$, we have the following equality in $(R\otimes_k R')[[X]][X^{-1}]$:
\begin{equation*}\label{gammaequalitynil:intro}
\gamma(\rho\otimes\rho',X,\psi) = \gamma(L_{R,n}(\rho)\times L_{R',t}(\rho'),X,\psi)
\end{equation*}
\end{enumerate}
\end{thm}
\noindent In fact, if all of the extensions $(L_{R,n})_{n\geq 1}$ exist, each one is unique (Corollary~\ref{compatibilitywithgammafactors}).

There remains the question of whether the sequence of bijections $(L_{k,n}^{\gen})_{n\geq 1}$ of Vign\'{e}ras actually admits the surjective extensions $L_{R,n}$ satisfying the conditions of Theorem~\ref{characterization}. The existence of the extensions satisfying conditions (1) and (2) of Theorem~\ref{characterization} follows from the machinery of the local Langlands correspondence ``in families," as it is stated in \cite{converse}, however its surjectivity is not immediate. This will be addressed in future work.

\subsection{Relation to other work and further questions}\label{otherwork}
Recently, the proof of Jacquet's conjecture was extended to the setting of reduced \emph{$\ell$-torsion free} $W(k)$-algebras in the preprint \cite{LM} of Liu and the author. As the counterexample in Section~\ref{counterexample} of the present article shows, the converse theorem in the $\ell$-torsion setting requires a different framework than \cite{LM}. 

In Section~\ref{vanishinglemmasgonewrong} we show that the point of failure in the traditional converse theorem proof method is the completeness of Whittaker models over $k$. We recover completeness of Whittaker models by including nilpotent lifts and using the geometry of the integral Bernstein variety, as developed in \cite{h_bern,h_whitt}.

The gamma factors used in Theorem~\ref{conversemain} are different from the ``new'' gamma factors proposed by Vign\'{e}ras in \cite{vig_epsilon}. Vign\'{e}ras' new gamma factors take coefficients in $k$, and to each cuspidal $\tau$ in $\mathcal{A}_k^{\gen}(2)$, $\chi$ in $\mathcal{A}_k^{\gen}(1)$, there is attached a collection of factors $\varepsilon(\tau\otimes\chi,y)$ which determine $\tau$ as $\chi$ and $y$ are allowed to vary ($y$ is an element of $\cO_F^{\times}$ of $\ell$-power order).

Given $\pi\in\mathcal{A}_k^{\gen}(n)$, $\tau\in\mathcal{A}_k^{\gen}(t)$, we speculate that there might be a way to construct a collection of gamma factors ${\gamma_1(\pi\times \tau,X,\psi),\gamma_2(\pi\times\tau,X,\psi),\dots}\in k(X)$ for which the mod-$\ell$ converse theorem holds, for which there is an analogous collection for $\mathcal{G}_k^{ss}(n)$, and for which there is an equality of sets 
\begin{align*}\{\gamma_1(L_{k,n}^{\gen}(\rho)\times L_{k,t}^{\gen}(\rho'),X,\psi), \gamma_2(L_{k,n}^{\gen}(\rho)\times L_{k,t}^{\gen}(\rho'),X,\psi),\dots\} \\=\{\gamma_1(\rho\otimes\rho',X,\psi), \gamma_2(\rho\otimes\rho',X,\psi),\dots\}.\end{align*} This would more closely resemble Vign\'{e}ras' construction of ``new'' gamma factors for $G_2$ in \cite{vig_epsilon}, and would eliminate the need for nilpotent lifts in Theorem~\ref{characterization}. It seems plausible that the set $\gamma_1(\pi\times \tau,X,\psi),\gamma_2(\pi\times\tau,X,\psi),\dots\in k(X)$ could be formed by taking appropriate linear combinations of the collection of $\gamma(\pi\times\widetilde{\tau},X,\psi)$ as $\widetilde{\tau}$ varies over nilpotent lifts of $\tau$, and mapping them to $k(X)$ in a clever way.

In the recent preprint \cite{km_ll}, Kurinczuk and Matringe demonstrate the failure of preservation of $L$-factors of pairs in $L_{k,n}^{\gen}$. We remark that the putative correspondence $L_{R,n}$ appearing in Theorem~\ref{characterization} would preserve gamma factors of pairs (and hence so would $L_{k,n}^{\gen}$), where the gamma factors are those of \cite{galois_gamma}.

The reason for assuming $F$ has characteristic zero is this paper's dependence on the results of \cite{h_bern,h_whitt,converse}, where the same assumption is made but almost certainly not required.

\subsection{Acknowledgements}
The ideas appearing here were influenced by several discussions with David Helm during and after the author's work on his PhD thesis. The interest and encouragement of Rob Kurinczuk and Nadir Matringe were invaluable. The author would also like to thank Jean-Fran\c{c}ois Dat, Guy Henniart, Marie-France Vign\'{e}ras, and an anonymous referee for their helpful comments at various stages.

\section{A counterexample to the naive mod-$\ell$ converse theorem}\label{counterexample}
In this section we give an example of two irreducible generic $k$-representations of $G_2:=GL_2(F)$ with distinct mod-$\ell$ supercuspidal supports (in fact, in different $\ell$-blocks), having the same gamma factors for all twists by characters. 

When writing \cite{vig_epsilon}, Vign\'{e}ras was no doubt aware of an example similar to the one presented below, but it does not appear in the literature. We also note that our counterexample is different from that in \cite{minguez}, where M\'{i}nguez gives two distinct irreducible $k$-represntations with the same Godement--Jacquet gamma factors, but one is nongeneric and they have the same supercuspidal support. 

In this section, we use the gamma factors as constructed in \cite{km}. Given $\pi$, $\tau$ in $\mathcal{A}_{k}^{\gen}(n)$, choose $\widetilde{\pi}$, $\widetilde{\tau}$ subrepresentations of $\Ind_U^G\psi_{\cO}$ lifting the Whittaker models of $\pi$, $\tau$, respectively. The classical gamma factor of \cite{jps2}, $\gamma(\widetilde{\pi}\otimes\cKb\times \widetilde{\tau}\otimes\cKb, X, \psi)$, defines a formal Laurent series with coefficients in $\cO$. Then $\gamma(\pi\times\tau,X,\psi)$ is defined as the reduction of $\gamma(\widetilde{\pi}\times\widetilde{\tau},X,\psi)$ modulo the maximal ideal of $\cO$. It is uniquely determined by a functional equation (c.f. \cite[Cor 3.11]{km}). Later, we will require gamma factors in a broader context, see \S~\ref{section:gamma}.

Let $q=p=5$, $\ell = 2$, and $n=2$. Let $\theta: \FF_{q^2}^{\times} \to \zlb^{\times}$ be the character that sends a primitive $24$'th root of unity to $\zeta_3$, a primitive $3$rd root of unity in $\zlb$. Since $\theta^q \neq \theta$, this is a regular character of $\FF_{q^2}^{\times}$, and therefore gives rise to an irreducible cuspidal representation $\lambda_{\theta}$ of $GL_2(\FF)$ (\cite[\S 6.4]{bh}). Let $K^0 = GL_2(\cO_F)$. Inflate $\lambda_{\theta}$ to a representation of $K^0$, then $\lambda_{\theta}$ has trivial central character, since $\theta|_{\FF_q^{\times}}$ is trivial (\cite[\S 6.4(1)]{bh}). Thus we may extend $\lambda_{\theta}$ to a representation $\Lambda$ of $F^{\times}K^0$ by declaring that $\Lambda|_{F^{\times}}$ is trivial. The triple $(M_2(\cO_F), F^{\times}K^0, \Lambda)$ is a \emph{cuspidal type} of level zero, in the language of \cite[15.5]{bh}, and the representation $\pi_{\theta}:= \cInd_{F^{\times}K^0}^{G_2}\Lambda$ is an irreducible cuspidal representation of $G_2$. In fact, $\pi_{\theta}$ is integral and, since $\theta^q\not\equiv\theta$ mod $\ell$, its mod-$\ell$ reduction $\overline{\pi_{\theta}}$ is supercuspidal, see \cite[III.3.3]{vig}, or \cite[Thm \S 2.7]{km}. 
 
\begin{lemma}\label{gausssum}
If $\chi:F^{\times}\to \cKb^{\times}$ is an unramified or a tamely ramified character, then $\gamma(\chi\pi_{\theta},X,\psi)\equiv 1$ mod $\ell$. If $\rm{level}(\chi)\geq 1$, then $\gamma(\chi\pi_{\theta},X,\psi)  = \gamma(\chi\circ \det,X,\psi)$.
\end{lemma}
\begin{proof}
If $\chi$ is tamely ramified, $\chi|_{U_F}$ is trivial mod-$\ell$, so there exists an unramified character $\chi'$ such that $\chi\equiv \chi'$ mod $\ell$. Hence $\chi\pi_{\theta}\equiv \chi'\pi_{\theta}$ mod $\ell$, and $\gamma(\chi\pi_{\theta},X,\psi) \equiv \gamma(\chi'\pi_{\theta},X,\psi)$ mod $\ell$ (\cite[Thm 3.13(2)]{km}). Thus we may assume without loss of generality that $\chi$ is unramified.

Since $\chi\pi_{\theta}$ is cuspidal, $\gamma(\chi\pi_{\theta},X,\psi) = \varepsilon(\chi\pi_{\theta},X,\psi)$. Since $\chi$ is unramified, the cuspidal type of $\chi\pi_{\theta}$ is the triple $(M_2(\cO_F), F^{\times}K^0, \chi\Lambda)$, and the restrictions $\chi\Lambda|_{K^0}$ and $\Lambda|_{K^0}$ are both equivalent to $\lambda_{\theta}$. It follows from the description in \cite[Section 25.4]{bh} that $\varepsilon(\chi\pi_{\theta},X,\psi)$ and $\varepsilon(\pi_{\theta},X,\psi)$ are equivalent and given by $-q\tau(\theta,\tilde{\psi})$, where $\tilde{\psi}(x) := \psi(x+x^q)$, $x\in \FF_{q^2}$, and $\tau(\theta,\tilde{\psi})$ is the Gauss sum $\sum_{x\in \FF_{25}^{\times}}\theta(x)\psi(x+x^q)$.

Since $x\mapsto x^q$ is a field automorphism, we have
\begin{align*}
\tau(\theta,\tilde{\psi}) &= \sum_{x\in \FF_{25}^{\times}}\theta(x^q)\psi(x^q+(x^q)^q)\\
&=\sum_{x\in \FF_{25}^{\times}}\theta(x)^q\psi(x+x^q)\\
&=\sum_{x\in \FF_{25}^{\times}}\theta(x)^2\psi(x+x^q)=\tau(\theta^{-1},\tilde{\psi}).
\end{align*}
Therefore, $\tau(\theta,\tilde{\psi})^2 = \tau(\theta,\tilde{\psi})\tau(\theta^{-1},\tilde{\psi})$ which we can compute as
$$
\sum_{x,y\in \FF_{25}^{\times}}\theta(xy^{-1})\tilde{\psi}(x+y)
= \sum_{u\in \FF_{25}^{\times}}\theta(u)\sum_{y\in \FF_{25}^{\times}}\tilde{\psi}(y(u+1)).
$$
Separating terms into $u=-1$ and $u\neq -1$,
\begin{align*}
\theta(-1)(25-1) + \sum_{u\neq -1}\theta(u)\sum_{y\in \FF_{25}^{\times}}\tilde{\psi}(y(u+1))&= \theta(-1)(25-1) + \sum_{u\neq -1}-\theta(u)\\
&= \theta(-1)(25) - \sum_{u\in \FF_{25}^{\times}}\theta(u)\\
&=\theta(-1)(25).
\end{align*}
It follows that $\varepsilon(\pi_{\theta},X,\psi)$ is congruent to $1$ mod $\ell$.

Since the central character of $\pi_{\theta}$ is trivial, the result for ramified characters of level $\geq 1$ follows immediately from the stability of gamma factors. For example, we can use the explicit formulation in \cite[25.7]{bh}, after observing that $\varepsilon(\chi\circ\det,X,\psi) = \gamma(\chi\circ\det,X,\psi)$ for characters of level $\geq 1$ (c.f. \cite[26.6 Prop]{bh}).
\end{proof}

Let $B$ denote the subgroup of upper triangular matrices of $G_2$, and $\textbf{1}:B\to k^{\times}$ the trivial character. The ``special" representation $\rm{Sp}_2$ is the cuspidal $k$-representation of $G_2$ that occurs as a subquotient of $i_B^G(\textbf{1})$ (note that normalized and non-normalized parabolic induction coincide since $q\equiv 1$ mod $\ell$). As $\rm{Sp_2}$ is a subquotient of the induction of an unramified character of the torus, its supercuspidal support is distinct (in fact, inertially inequivalent) from $\overline{\pi_{\theta}}$.

\begin{lemma}\label{Lfactor}
If $\bar{\chi}:F^{\times}\to k^{\times}$ is an unramified or tamely ramified character, $\gamma(\bar{\chi}\rm{Sp}_2,X,\psi)\equiv 1$ mod $\ell$. If $\rm{level}(\bar{\chi})\geq 1$, then $\gamma(\bar{\chi}\rm{Sp}_2,X,\psi) = \gamma(\bar{\chi}\circ \det,X,\psi)$. 
\end{lemma}
\begin{proof}
By the multiplicativity of gamma factors, 
\begin{align*}
\gamma(\bar{\chi}\rm{Sp_2},X,\psi)&=\gamma(\bar{\chi}\circ\det,X,\psi)\\
&= \gamma(i_B^G\bar\chi,X,\psi)\\
&= \gamma(\bar{\chi},X,\psi)^2.
\end{align*}
If $\bar{\chi}$ is unramified (which is equivalent to tamely ramified since $q-1$ is a power of $\ell$), we may choose an unramified lift $\chi:F^{\times}\to\cKb^{\times}$ and compute
\begin{align*}
\gamma(\chi,X,\psi) &= \varepsilon(\chi,X,\psi)\frac{L(\chi^{-1},1/(q^{1/2}X))}{L(\chi,q^{-1/2}X)}\\
&= \chi(\varpi)^{-1}X^{-1}\frac{1-\chi(\varpi)q^{-1/2}X}{1-\chi(\varpi)^{-1}q^{-1/2}X^{-1}}\\
&= \frac{\chi(\varpi)^{-1}X^{-1}-q^{-1/2}}{1-\chi(\varpi)^{-1}q^{-1/2}X^{-1}},
\end{align*}
which is equivalent to $-1$ mod $\ell$, since $q\equiv 1$ mod $\ell$.

It remains to show that $\gamma(\bar{\chi}\circ \det,X,\psi) = \varepsilon(\bar{\chi}\circ\det,X,\psi)$ for all $\chi$ of level $\geq 1$. This is shown in \cite[\S 6]{minguez}.
\end{proof}

\begin{cor}
For every character $\bar{\chi}:F^{\times}\to k^{\times}$, $\gamma(\bar{\chi}\overline{\pi_{\theta}},X,\psi)=\gamma(\bar{\chi}\rm{Sp}_2,X,\psi)$, but $\overline{\pi_{\theta}}$ and $\rm{Sp}_2$ have distinct supercuspidal supports (in fact, inertially distinct).
\end{cor}

\begin{rmk}
There are analogues of the local converse theorem for representations of $GL_n(\mathbb{F}_q)$ over $\mathbb{C}$ (\cite{finite_field_jacquet}). Granted the extension of the theory of $GL_n(\mathbb{F}_q)$ gamma factors to $k$-representations, the Gauss sum calculations in this example could be adapted to illustrate the failure of the naive mod-$\ell$ converse theorem over finite fields.
\end{rmk}
\begin{rmk}
Finding similar examples when $\ell\neq 2$ appears to be nontrivial. It could be interesting to describe the ``gamma-factor $\ell$-blocks,'' i.e. the sets of inertial supercuspidal supports whose twisted gamma factors become equivalent upon reduction modulo $\ell$.
\end{rmk}

The lack of tamely ramified $k$-valued characters was important in establishing the congruences in Lemmas~\ref{gausssum}, \ref{Lfactor}. Let $\cO_0$ be the sub-$W(k)$-algebra of $\zlb$ generated by the values of a tamely ramified $\chi$, which sends a primitive $24$'th root of unity in $\FF_{25}^{\times}$ to a primitive $4$'th root of unity. In particular, $\cO_0$ is a obtained from $W(k)$ by adjoining a $4$'th root of unity, or equivalently a $4$'th root of $\ell$, and therefore $R=\cO_0/\ell\cO_0$ is isomorphic to the four-dimensional local $k$-algebra $k[Y]/Y^4$, and $\zeta:=Y+1$ is a fourth root of unity such that $\zeta^2\neq 1$. Now let $\overline{\chi}$ be the reduction mod $\ell\cO_0$ of $\chi$, \emph{instead} of the reduction modulo the maximal ideal of $\cO_0$. Now $\overline{\chi}$ is a nilpotent lift of the reduction of $\chi$ modulo the maximal ideal of $\cO_0$. With the notation of \S~\ref{counterexample}, $\gamma(\overline{\chi}\overline{\pi_{\theta}},X,\psi)$ is given by the reduction mod $\ell$ of $-q\tau(\chi_E\theta,\tilde{\psi})$, where $\chi_E(x):=\chi(x^{q+1})$ on $\FF_{25}^{\times}$. The same calculation as in the proof of Lemma~\ref{gausssum} gives $\gamma(\overline{\chi}\overline{\pi_{\theta}},X,\psi)^2 = \overline{\chi}(-1)=\zeta^2\neq 1$, where we consider $-1$ as an element of $\FF_{25}^{\times}$. On the other hand, the same calculation as in the proof of Lemma~\ref{Lfactor} shows that $\gamma(\overline{\chi}\text{Sp}_2,X,\psi)$ is the reduction of $\tau(\chi_E\theta,\tilde{\psi})^2=\chi(-1)$ (\cite[23.6.2]{bh}). Thus, $\gamma(\overline{\chi}\overline{\pi_{\theta}},X,\psi)^2= \gamma(\overline{\chi}\text{Sp}_2,X,\psi) = \zeta^2$. It follows that $\gamma(\overline{\chi}\overline{\pi_{\theta}},X,\psi)\neq \gamma(\overline{\chi}\text{Sp}_2,X,\psi)$. This illustrates, for our particular example, how finite-dimensional $k$-algebras are large enough to distinguish twisted gamma factors in characteristic $\ell$.

\section{Co-Whittaker modules and the integral Bernstein center}
\label{section:cowhittaker}

To recover the converse theorem in characteristic $\ell$, we will pass to $R$-coefficients, where $R$ is an Artin local $k$-algebra. For this we need the theory of co-Whittaker $R[G_n]$-modules, where $R$ is a $W(k)$-algebra. 

If $V$ is a smooth $R[G_n]$-module, define $V^{(n)}$ to be the $\psi$-coinvariants $V/V(U_n,\psi)$, where $V(U_n,\psi)$ is the $R$-submodule generated by $\{\psi(u)v-uv:u\in U_n,v\in V\}$. This functor is exact and, for any $R$-module $M$ there is a natural isomorphism 
$$(V\otimes_RM)^{(n)}\cong V^{(n)}\otimes_RM\,.$$ If $V^{(n)}$ is nonzero, Frobenius reciprocity produces a canonical nonzero map 
\begin{equation}\label{canonicalwhittaker}
V\to \Ind_U^G\psi_{V^{(n)}}:v\mapsto W_v,
\end{equation} where $\psi_{V^{(n)}}:=\psi\otimes_{W(k)}V^{(n)}$. The quotient map $V\to V^{(n)}$ is equivalent to $v\mapsto W_v(1)$.

\begin{defn}
Let $R$ be a Noetherian $W(k)$-algebra. A smooth $R[G_n]$-module $V$ is co-Whittaker if the following conditions hold
\begin{enumerate}
\item $V$ is admissible as an $R[G_n]$-module,
\item $V^{(n)}$ is a free $R$-module of rank one,
\item if $Q$ is a quotient of $V$ such that $Q^{(n)}=0$, then $Q=0$.
\end{enumerate}
\end{defn}

For example, when $R=\CC$, $n=2$, $B$ is the Borel subgroup, and $\chi=\chi_1\otimes\chi_2$ is a character of the torus $T$, the normalized parabolic induction $i_B^G(\chi)$ is co-Whittaker so long as $\chi_1\chi_2^{-1}\neq |\cdot |$. 

If $V$ and $V'$ are co-Whittaker $R[G]$-modules, any nonzero $G$-equivariant map $V\to V'$ is surjective, as otherwise the cokernel would be a nongeneric quotient. In this case $V$ is said to dominate $V'$. We say $V$ and $V'$ are {\it equivalent} if there exists a co-Whittaker $R[G_n]$-module $V''$ dominating both $V$ and $V'$. This is an equivalence relation on isomorphism classes of co-Whittaker modules.

By definition a co-Whittaker module admits an isomorphism $V^{(n)}\cong R$, which induces an isomorphism $\Ind_U^G\psi_{V^{(n)}}\cong \Ind_U^G\psi_R$. The image of $V$ in the composition $$V\to \Ind_U^G\psi_{V^{(n)}}\cong \Ind_U^G\psi_R$$ will be denoted by $\cW(V,\psi)$ and called the ($R$-valued) \emph{Whittaker model of $V$} with respect to $\psi$. It is independent of the choice of isomorphism $V^{(n)}\cong R$.

\begin{lemma}
If $V$ is a co-Whittaker $R[G_n]$-module, its Whittaker model $\cW(V,\psi)$ is an equivalent co-Whittaker $R[G_n]$-module.
\end{lemma}
\begin{proof}
Since $\cW(V,\psi)$ is a quotient of $V$, we need only show $\cW(V,\psi)^{(n)}$ is free of rank one. Let $\cW$ be the image of $V$ in the map~(\ref{canonicalwhittaker}). Choosing an isomorphism $V^{(n)}\cong R$ induces an isomorphism $\cW\cong \cW(V,\psi)$, so it suffices to prove $\cW^{(n)}\cong V^{(n)}$. But the map $V\to V^{(n)}$ factors as $V\to \cW \to V^{(n)}$, the second map being evalution at 1. This induces a map $\cW^{(n)}\to V^{(n)}$, which is surjective. On the other hand, the natural surjection $V\to \cW$ induces a surjection $V^{(n)}\to \cW^{(n)}$ which is its inverse. 
\end{proof}

In \cite{h_whitt}, Helm constructs a co-Whittaker module which is ``universal'' up to this notion of equivalence. The key tool is the integral Bernstein center of $G_n$, i.e. the center of the category $\Rep_{W(k)}(G_n)$.

The center of an abelian category is the endomorphism ring of the identity functor, in other words the ring of natural transformations from the identity functor to itself. It acts on every object in the category in a way compatible with all morphisms. We denote by $\cZ_n$ the center of $\Rep_{W(k)}(G_n)$. 

Schur's lemma holds for any co-Whittaker $R[G_n]$-module $V$, meaning the natural map $R\to \End_{A[G_n]}(V)$ is an isomorphism (c.f. \cite[Prop 6.2]{h_whitt}), and thus there exists a map $f_V:\cZ_n \to R$, which we call the \emph{supercuspidal support} of $V$. Note that $V$ also admits a central character $\omega_V:F^{\times}\to R^{\times}$. If $R$ is a field and $V, V'$ are objects in $\mathcal{A}^{\gen}_{R}(n)$, then $f_V=f_{V'}$ if and only $\scs(V) = \scs(V')$ in the traditional sense (\cite[12.12]{h_bern}).

A primitive idempotent $e$ of $\cZ_n$ gives rise to a direct factor category $e\Rep_{W(k)}(G_n)$, which is the full subcategory of $\Rep_{W(k)}(G_n)$ on which $e$ acts as the identity. As described in \cite{h_bern}, the primitive idempotents in $\cZ_n$ are in bijection with inertial equivalence classes of pairs $(L,\pi)$, where $L$ is a Levi subgroup of $G_n$ and $\pi$ is an irreducible supercuspidal $k$-representation of $L$.

If $V$ is a simple object in $\Rep_{W(k)}(G_n)$ that is killed by $\ell$, the mod-$\ell$ inertial supercuspidal support is defined to be the usual inertial supercuspidal support of $V$ in $\Rep_k(G_n)$. If $V$ is an integral simple object in $\Rep_{\overline{\cK}}(G_n)$, its mod-$\ell$ reduction may no longer be simple. The mod-$\ell$ inertial supercuspidal support of $V$ is defined by combining the inertial supercuspidal supports of the constituents of its mod-$\ell$ reduction (c.f. \cite[4.12,4.13]{h_bern}).

If $e$ is the idempotent corresponding to the pair $(L,\pi)$, then a representation $V$ in $\Rep_{W(k)}(G_n)$ lies in $e\Rep_{W(k)}(G_n)$ if and only if every simple subquotient of $V$ has mod-$\ell$ inertial supercuspidal support $[L,\pi]$.

If $[M',\pi']$ is an inertial cuspidal support in $\Rep_{\overline{\cK}}(G_n)$, let $\cZ_{M',\pi'}$ denote the center of the subcategory of $\Rep_{\overline{\cK}}(G_n)$ corresponding to $[M',\pi']$.
\begin{thm}[\cite{h_bern}, Thm 10.8, 12.1]\label{characteristiczerosplitting}
Let $e$ be any primitive idempotent of $\cZ_n$, corresponding to a mod-$\ell$ inertial supercuspidal support $[L,\pi]$.
\begin{enumerate} 
\item The ring $e\cZ_n$ is a finitely generated, reduced, $\ell$-torsion free $W(k)$-algebra.
\item There is an isomorphism $$e\cZ\otimes\overline{\cK}\cong \prod_{[M',\pi']}\cZ_{M',\pi'},$$ where $[M',\pi']$ varies over pairs consisting of a Levi $M'\supset M$ and $\pi'$ a cuspidal integral object in $\Rep_{\overline{\cK}}(M')$ with mod-$\ell$ inertial supercuspidal support equivalent to $[L,\pi]$.
\end{enumerate}
\end{thm}

The following example indicates that, in characteristic $\ell$, valuable information is lost by restricting one's attention only to the set of $k$-points $e\cZ_n \to k$ rather than, say, the set of points $e\cZ_n \to k[\varepsilon]/(\varepsilon)^2$.
\begin{ex}\label{lpowerrootsbernstein}
Let $n=2$, $\ell=3$, $q=5$, and let $e$ be the primitive idempotent corresponding to the mod-$\ell$ inertial supercuspidal support $[T,\textbf{1}]$, where $T$ is the maximal torus and $\textbf{1}$ is its trivial character. If $\zeta$ is a nontrivial cube root of unity in $\cKb$, then $\zeta + \zeta^{-1}=-1$, and it follows from \cite[13.11]{h_bern}, combined with \cite[Thm 4.11]{paige} that $$e\cZ_2 \cong W(k)[\Theta_1,\Theta_2^{\pm 1},Y]/\left((Y-2)\Theta_1, (Y-2)(Y+1)\right).$$ In particular, its spectrum possesses two irreducible components: $\{Y=2\}$ and $\{\Theta_1=0\}$, which become the principal block and a cuspidal block, respectively, over $\cKb$ in the sense of Theorem~\ref{characteristiczerosplitting}. On the other hand, since $-1\equiv 2$ mod $\ell$, the special fiber $e\cZ_2/(\ell)$ contains the nontrivial cube-root of unity $\zeta=Y-1$.
\end{ex}

Now let $R$ be a $W(k)$-algebra, and let $V$ be a co-Whittaker $R[G_n]$-module.  Suppose further that $V$ lies in $e\Rep_{W(k)}(G_n)$ for some primitive idempotent $e$ (so the supercuspidal support map $f_V$ factors through the projection $\cZ_n\to e\cZ_n$). Let $\fW_n$ be the smooth $W(k)[G_n]$-module $\cInd_{U_n}^{G_n} \psi$.  For any
primitive idempotent $e$ of $\cZ_n$, we have an action of $e\cZ_n$ on $e\fW_n$.

\begin{thm}[\cite{h_whitt}, Theorem 6.3]
\label{universalcowhitt} 
Let $e$ be any primitive idempotent of $\cZ_n$.
The smooth $e\cZ_n[G_n]$-module $e\fW_n$
is a co-Whittaker $e\cZ_n[G_n]$-module. If $R$ is Noetherian and has an $e\cZ_n$-algebra structure, the module $e\fW_n \otimes_{e\cZ_n} A$
is a co-Whittaker $R[G_n]$-module. Conversely, 
$V$ is dominated by $e\fW_n \otimes_{e\cZ_n, f_V} R$.
\end{thm}

We thus say that, up to the equivalence relation induced by dominance, $e\fW_n$ is
the universal co-Whittaker module in $e\Rep_{W(k)}(G_n)$.

\begin{lemma}[\cite{LM} Lemma 2.4]
\label{equivalencecriterion}
Let $R$ be a Noetherian ring and suppose $V_1$ and $V_2$ are two co-Whittaker $R[G_n]$-modules. The following are equivalent:
\begin{enumerate}
\item There exists $W$ in $\cW(V_1,\psi)\cap \cW(V_2,\psi)$ such that $W(g)\in R^{\times}$ for some $g\in G$.
\item $\cW(V_1,\psi)= \cW(V_2,\psi)$.
\item $f_{V_1}= f_{V_2}$.
\item $V_1$ and $V_2$ are equivalent.
\end{enumerate}
\end{lemma}

\begin{cor}\label{nilpotentliftcowhitt}
Let $R$ be a complete local Noetherian $W(k)$-algebra with residue field $k$ (e.g. an Artin local $k$-algebra). There is a bijection from the set of equivalence classes of co-Whittaker $R[G_n]$-modules to the set $\mathcal{A}_{R}^{\gen}(n)$, sending an equivalence class to the Whittaker model of any representative.
\end{cor}
\begin{proof}
If $\widetilde{\tau}$ is in $\mathcal{A}_{R}^{\gen}(n)$, we must show it is co-Whittaker. But this follows immediately from \cite[6.3.2 Lemma, 6.3.4 Proposition]{eh}.

Conversely, let $V$ be a representative of an equivalence class of co-Whittaker $R[G_n]$-modules. By Lemma~\ref{equivalencecriterion}, each equivalence class has a unique Whittaker model $\cW\subset \Ind_U^G\psi_{R}$. The cosocle (i.e. largest semisimple quotient) of $V\otimes_Rk$ is irreducible and generic (e.g. \cite[6.3.5 Lemma]{eh}). Hence $\text{cosoc}(V\otimes_Rk)$ isomorphic to its Whittaker model. Thus the reduction $\cW\otimes_{R}k$ is isomorphic to the Whittaker model of $\text{cosoc}(V\otimes_Rk)$, so $\cW$ is a nilpotent lift of $\text{cosoc}(V\otimes_Rk)$.
\end{proof}

There is a duality operation on co-Whittaker modules which interpolates the contragredient across a co-Whittaker family (\cite[Prop 2.6]{converse}). If $V$ is a smooth $W(k)[G_n]$-module, let $V^{\iota}$ denote the $W(k)[G_n]$-module with the same underlying $W(k)$-module structure, and for which the $G_n$ action, which we will denote by $g\cdot v$, is given by $g\cdot v = g^{\iota}v$, where $g^{\iota}={ ^tg}^{-1}$. This duality has a very concrete interpretation in terms of Whittaker functions. Let 
$$\omega_{n,m} = \begin{pmatrix}
I_{n-m}& 0 \\
0 & \omega_m
\end{pmatrix},$$
 where $\omega_m$ is the longest Weyl element of $G_m$. For any function $W$ on $G_n$, let $\widetilde{W}(g)=W(\omega_n g^{\iota})$. If $W$ is in $\cW(V,\psi)$, then $\widetilde{W}$ is in $\cW(V^{\iota},\psi^{-1})$.

\section{Failure of completeness of Whittaker models over $\flb$}\label{vanishinglemmasgonewrong}
Over $\CC\cong \cKb$, if $H$ is in $\cInd_{U_t}^{G_t}\psi_{\cKb}$, and $$\langle H, W\rangle:= \int_{U\backslash G}H(x)W(x)dx = 0$$ for all Whittaker functions $W$ in $\cW(\tau,\psi^{-1})$, for all $\tau\in \mathcal{A}_{\cKb}^{\gen}(t)$, then $H=0$. In other words, the Whittaker models of $\mathcal{A}_{\cKb}^{\gen}(t)$ are \emph{complete} with respect to this integral pairing. In all known proofs of converse theorems, completeness of Whittaker models is used in an essential way.

\begin{question}\label{modellvanishingquestion}
Let $H$ be an element of $\cInd_{U_t}^{G_t}\psi_k$. Suppose that $\int_{{U_t}\backslash G_t}H(x)W(x)dx =0$ for all $W\in \cW(\tau,\psi_k^{-1})$ for all $\tau\in\mathcal{A}_t^{\gen}(t)$. Must $H$ be zero?
\end{question}
The answer to Question~\ref{modellvanishingquestion} is ``no," as the following example shows already when $t=1$. Suppose that $q\equiv 1\mod \ell$ and $q-1 = \ell^am$ where $m$ is relatively prime to $\ell$. Decompose $\cO_F^{\times} = X\times Y$ where $X$ is subgroup of $\cO_F^{\times}$ whose pro-order is prime-to-$\ell$ and $Y$ is a cyclic group of order $\ell^a$. Since $q$ is invertible in $k$, there exists a $k$-valued Haar measure $\mu$ on $F^{\times}$, which we may normalize so that $\mu(X)=1$ (or any unit). We remark that $\mu(\cO_F^{\times})=0$. Let $n=1$, so that $U=\{1\}$ and $\psi$ is trivial, and every smooth character $\chi:F^{\times}\to k^{\times}$ is irreducible generic with Whittaker space $c\cdot \chi$ for constants $c\in k$. Take $H$ in $\cInd_{U_1}^{G_1}\psi = C_c^{\infty}(F^{\times},W(k))$ to be the characteristic function of $\cO_F^{\times}$. Given a character $\chi$, we have
$$\langle H, \chi\rangle = \int_{F^{\times}}H(x)\chi(x)dx = \int_{\cO_F^{\times}}\chi(x)dx.$$ Since each $y\in Y$ has $\ell$-power order, so does $\chi(y)\in k$, and thus $\chi(y)=1$. Let $K$ be the largest finite index subgroup of $X$ such that $K\subset \ker(\chi)$. Then for any $\chi$ we have
$$\langle H, \chi\rangle=\mu(K)\sum_{x\in \cO_F^{\times}/K}\chi(x)=\ell^a[X:K]^{-1}\sum_{x\in X/K}\chi(x)=0.$$

However, if we expand the coefficient ring of $\tau$ enough, the answer becomes ``yes", as we will now describe. 

The property of being a co-Whittaker model does not depend on the choice of $\psi$ in the definition of $(-)^{(n)}$. The Whittaker model of $e\fW_n$ with respect to $\psi^{-1}$ has the following $W(k)$-module structure, as observed in \cite[p. 1010]{converse}.

\begin{prop}[\cite{converse}]
\label{universalwhittakerfunctions}
There is a map of $W(k)$-modules $\theta_e:e\cZ_n \to W(k)$ which induces an isomorphism of $W(k)$-modules 
\begin{align*}
e\fW_n &\isomto \cW(e\fW_n, \psi^{-1})\\
f &\mapsto W_f\\
\theta_e\circ W &\mapsfrom W
\end{align*}
\end{prop}

The following Corollary is essentially a duality statement:

\begin{cor}
\label{bigvanishinglemmaarbitrary}
Let $R$ be a $W(k)$-algebra and suppose $H$ is a nonzero element of $\cInd\psi_R$ Then there is a primitive idempotent $e$ of $\cZ$ and a Whittaker function $W$ in $\cW(e\fW_n, \psi^{-1})$ such that 
$$\langle H,W\rangle:= \int_{U\backslash G}H(x)\otimes W(x)dx$$ is nonzero in $R\otimes_{W(k)}e\cZ$.
\end{cor}
\begin{proof}
As a function on $G$ we have $H(g)\neq 0$ in $R$ for some $g\in G$. Now choose some compact open subgroup $K$ such that $H$ is invariant under $K$ (note: we must take $K$ to be small enough that its Haar measure is nonzero) and such that $\psi$ is trivial on $U\cap gKg^{-1}$. We build the function $\phi$ in $\cInd_{U}^G\psi^{-1}$ as follows. $\phi$ will be supported on the double coset $UgK$ and will satisfy $\phi(ngk)=\frac{1}{\text{meas}(K)}\psi(n)^{-1}$. In particular, $\langle H, \phi\rangle=H(g)$ is nonzero in $R$. Now as $\cInd_U^G\psi^{-1}$ is the direct sum of $e(\cInd\psi^{-1})$ for primitive idempotents $e$, there must be some primitive idempotent $e$ such that $\langle H, e\phi\rangle \neq 0$.

In the isomorphism of Proposition \ref{universalwhittakerfunctions} there is a Whittaker function $W_{e\phi}\in \cW(e\fW_n,\psi^{-1})$ corresponding to the function $e\phi$, and we have $\theta_e(W_{e\phi}(x))=e\phi(x)$. Now $\langle H,W_{e\phi}\rangle$ cannot be zero, since otherwise its image $\langle H, e\phi\rangle$ under the map $\rm{id}\otimes\theta_e$ would be zero.
\end{proof}

\section{Completeness of Whittaker models over Artin local $k$-algebras}

In the previous section we showed that completeness of Whittaker models holds after allowing representations with coefficients in $\cZ$, which is quite large. The main result of this paper is that it suffices to allow coefficients only in finite-dimensional local $k$-algebras.

\begin{thm}\label{vanishingthm}
Let $R$ be an Artin local $k$-algebra. Let $H$ be an element of $\cInd_{U_t}^{G_t}\psi_{R}$. Then the following implication holds:
if $\langle H,W\rangle =0$ for all $W\in\cW(\widetilde{\tau},\psi)$, for all $\widetilde{\tau}\in \mathcal{A}^{\gen}_{\nil}(t)$, then $H=0$.
\end{thm}

We state the following special case of Theorem~\ref{vanishingthm}, where $H$ takes values in $k$.

\begin{cor}\label{corollaryoverk}
Let $H$ be an element of $\cInd_{U_t}^{G_t}\psi_{k}$. Then the following implication holds:
if $\langle H,W\rangle =0$, for all $W\in\cW(\widetilde{\tau},\psi)$, for all $\widetilde{\tau}\in \mathcal{A}^{\gen}_{\nil}(t)$, then $H=0$.
\end{cor}

The proof of Theorem~\ref{vanishingthm} uses a simple geometric argument on the mod-$\ell$ Bernstein center, which has a subtlety arising from the presence of nilpotents. Consider, for the sake of illustration, the special case where $H$ takes values in $k$. To prove Theorem~\ref{vanishingthm}, we will prove the contrapositive. If $H\neq 0$, consider the Whittaker function $W$ guaranteed by Lemma~\ref{bigvanishinglemmaarbitrary}, so that $\langle H,W\rangle$ is nonzero in $e\cZ\otimes k$. If we knew that $\langle H,W\rangle$ were non-nilpotent, the principal open set $D(\langle H, W\rangle)$ would be nonempty, and would necessarily intersect the dense set of closed points, in other words there is a map $e\cZ\otimes k\to k$ that does not kill $\langle H,W\rangle $. However, $e\cZ\otimes k$ has many nilpotents in general, and $\langle H, W\rangle$ may be among them. Thus, instead of closed points, we instead look for points $e\cZ\otimes k \to R$ to non-reduced finite-dimensional $k$-algebras $R$ in which the image of $\langle H, W\rangle$ is nonzero.

The author is grateful to David Helm for communicating the following lemma.
\begin{lemma}\label{powerofmaximalideal}
Let $S$ be a Noetherian ring, and let $x$ be a nonzero element of $S$. Then there exists a maximal ideal $\fm$ of $S$ and an integer $i$ such that $x$ is not in $\fm^i$.
\end{lemma}
\begin{proof}
We would like to show that $M = \bigcap_{\fm\text{ max'l}}\left(\bigcap_{i\geq 1}\fm^i\right)$ is zero. For a fixed $\fm$, there is an inclusion $M\hookrightarrow \bigcap_{i}\fm^i$, $i\geq 1$. The $S$-module $\bigcap_{i}\fm^i$ is finitely generated over $S$ and satisfies $$\fm\cdot\left(\bigcap_i\fm^i\right)=\bigcap_i\fm^i.$$ The same is true of its localization $\left(\bigcap_i\fm^i\right)_{\fm}$ at the ideal $\fm$. It follows by Nakayama that the localization $\left(\bigcap_i\fm^i\right)_{\fm}$ is zero. But the map on localizations $M_{\fm}\to \left(\bigcap_i\fm^i\right)_{\fm}$ remains injective, and we conclude that $M_{\fm}=0$.
\end{proof}

\begin{proof}[Proof of Thm~\ref{vanishingthm}]
Take $S$ in Lemma~\ref{powerofmaximalideal} to be $e\cZ\otimes_{W(k)}R$, and take $x\in S$ to be the element $\langle H,W\rangle$, where $W$ is the Whittaker function guaranteed by Lemma~\ref{bigvanishinglemmaarbitrary}. In particular, $x$ is nonzero. By Lemma~\ref{powerofmaximalideal}, there is a maximal ideal $\fm$ of $S$ such that $x$ is not killed in the map $S\to S/\fm^i$. 

The ring $S/\fm^i$ has a single prime ideal, $\fm/\fm^i$, which is the nilradical of $S/\fm^i$. Since $S$ is a finitely generated $k$-algebra, $S/\fm = k$. The successive quotients of the filtration $S/\fm^i \supset \fm/\fm^i\supset \dots \supset \fm^i/\fm^i = 0$ are one-dimensional $k$-vector spaces, therefore $S/\fm^i$ has length $i$ as a module over itself. In particular $S/\fm^i$ is an Artinian local $k$-algebra.

Let $\phi:S\to S/\fm^i$ be the quotient map, let $f$ be the composition $e\cZ\to S \to S/\fm^i$ and let $g$ be the composition $R\to S\to S/\fm^i$. Then $\langle H, g\circ W\rangle$ must be nonzero in $(S/\fm^i)\otimes_{W(k)}R$, since otherwise its image $$\langle f\circ H, g\circ W\rangle = \phi(\langle H,W\rangle)$$ in $S/\fm^i$ would have to be zero, contrary to our construction of $\phi$. The Theorem now follows since $g\circ W$ is a Whittaker function in the $S/\fm^i$-valued Whittaker space of $\widetilde{\tau}:=e\fW\otimes_{e\cZ,f}(S/\fm^i)$, which is a co-Whittaker $(S/\fm^i)[G]$-module (c.f. Corollary~\ref{nilpotentliftcowhitt}).
\end{proof}

\section{Rankin--Selberg theory and gamma factors}
\label{section:gamma}

This section closely follows \cite{moss2}, and the reader should refer there for more details. Let $A$ and $B$ be Noetherian $W(k)$-algebras and let $R = A\otimes_{W(k)}B$. Let $V$ and $V'$ be co-Whittaker $A[G_n]$- and $B[G_t]$-modules respectively, where $t<n$, with central characters $\omega_V$, $\omega_{V'}$. For $W\in \cW(V,\psi)$ and $W'\in \cW(V', \psi^{-1})$, 
and for $0\leq j \leq n-t-1$, we define the formal series in the variable $X$ with coefficients in $R$: 
\begin{align*}
& \ \Psi(W,W',X;j)\\
:= & \ \sum_{r\in \ZZ}\int_{M_{j,t}(F)}\int_{U_t\backslash\{g\in G_t:v(\det g)=r\}}                    \left(W\left(\begin{smallmatrix}g &&\\x&I_j &\\&&I_{n-t-j} \end{smallmatrix}\right)\otimes W'(g)\right)X^rdg dx\,.
\end{align*}

It is straightforward to show that there are finitely many negative-power terms, so $\Psi(W,W',X):=\Psi(W,W',X;0)$ defines an element of $R[[X]][X^{-1}]$. 

Let $S$ be the multiplicative system in $R[X, X^{-1}]$ consisting of all polynomials whose leading and trailing coefficients are units.

\begin{thm}[\cite{moss2}]
\label{fnleqn}
\begin{enumerate}
\item Both $\Psi(W,W',X)$ and $\Psi(W,W',X;j)$ are elements in $S^{-1}(R[X, X^{-1}])$. 

\item There exists a unique element $\gamma (V \times V', X, \psi) \in S^{-1}(R[X, X^{-1}])$ such that 
\begin{align*}
& \ \Psi(W, W', X; j) \gamma(V \times V', X, \psi) \omega_{V'}(-1)^{n-1} 
\\
= & \ \Psi(\omega_{n,t}\widetilde{W}, \widetilde{W'}, \frac{q^{n-t-1}}{X}; n-t-1-j)\,,
\end{align*}
for any $W \in \cW(V,\psi)$, $W' \in \cW(V', \psi)$ and for any $0 \leq j \leq n-t-1$. 
\end{enumerate}
\end{thm}

\begin{prop}[\cite{LM} Prop 3.3]\label{centralcharacter}
Let $A$ be a Noetherian $W(k)$-algebra, let $\pi_1, \pi_2$ be co-Whittaker $A[G_n]$-modules with $n\geq 2$. Assume that 
$$\gamma(\pi_1\times \chi,X,\psi)=\gamma(\pi_2 \times \chi,X,\psi)\,,$$
for any character $\chi:F^{\times}\to W(k)^{\times}$. 
Then $\omega_{\pi_1} = \omega_{\pi_2}$. 
\end{prop}

\section{The mod-$\ell$ converse theorem}

In this section we deduce the following mod-$\ell$ converse theorem.

\begin{thm}
\label{conversefull}
Let $R$ be an Artin local $k$-algebra, let $n\geq 2$, and let $\widetilde{\pi_1}$, $\widetilde{\pi_2}$ be in $\mathcal{A}_{R}^{\gen}(n)$.
If $$\gamma(\widetilde{\pi_1}\times\widetilde{\tau},X,\psi) = \gamma(\widetilde{\pi_2}\times\widetilde{\tau},X,\psi)$$ for all nilpotent lifts $\widetilde{\tau}\in \mathcal{A}_{\nil}^{\gen}(t)$, for all $1\leq t\leq \lfloor \frac{n}{2}\rfloor$, then $\widetilde{\pi_1}=\widetilde{\pi_2}$.
\end{thm}

Given $\pi_1$, $\pi_2$ in $\mathcal{A}_{k}^{\gen}(n)$, we can identify them with their trivial nilpotent lifts (i.e. their Whittaker models), which gives Theorem~\ref{conversemain} from the introduction.

Note that Proposition~\ref{centralcharacter} removes the need (when $n\geq 2$) for the hypothesis that $\widetilde{\pi_1}$ and $\widetilde{\pi_2}$ have the same central character.

\begin{proof}[Proof of Theorem~\ref{conversefull}]
The only part of the proof of the main theorem in \cite{LM} that requires reducedness and $\ell$-torsion free-ness is the completeness of Whittaker models statement \cite[Thm 4.1]{LM}. After replacing Theorem 4.1 in \cite{LM} with Theorem~\ref{vanishingthm} of the present paper, and applying the argument of \cite[\S 4-6]{LM} mutatis mutandis, we obtain Theorem~\ref{conversefull}.
\end{proof}

\begin{rmk}\label{rmk:restrictrings}
From the proof of Theorem~\ref{vanishingthm}, we see that in Theorems~\ref{vanishingthm} and ~\ref{conversefull} we can restrict the range of $\widetilde{\tau}$ to $\mathcal{A}_{R'}^{\gen}(t)$, for rings $R'$ specifically of the form $$(e\cZ\otimes_{W(k)}R)/\fm^i,$$ where $e$ is a primitive idempotent of $\cZ$, $\fm$ is a maximal ideal of $e\cZ\otimes_{W(k)}R$, and $i$ is a positive integer. Maybe it is possible to classify such rings.
\end{rmk}
\section{Deligne--Langlands gamma factors of nilpotent lifts}
If $\kappa$ is a field of characteristic zero containing $W(k)$, and $\rho: W_F \rightarrow GL_n(\kappa)$ is a Weil group representation,
then there is a rational function $\gamma(\rho, X, \psi)$ in $\kappa(X)$, called the Deligne--Langlands $\gamma$-factor
of $\rho$.  The main result of~\cite{galois_gamma} extends this construction to families of Weil group representations.  In particular, one has:
\begin{thm}[\cite{galois_gamma}, Theorem 1.1] \label{thm:local galois gamma}
Let $R$ be a Noetherian $W(k)$-algebra and let $\rho: W_F \rightarrow GL_n(R)$ be a representation that is
$\ell$-adically continuous in the sense of~\cite[\S 2]{galois_gamma}.
Then there exists an element $\gamma_R(\rho, X, \psi)$ of $S^{-1}R[X,X^{-1}]$ with the following properties:
\begin{enumerate}
\item If $f: R \rightarrow R'$ is a homomorphism of Noetherian $W(k)$-algebras, then one has:
$$f(\gamma_R(\rho, X, \psi)) = \gamma_{R'}(\rho \otimes_R R', X, \psi),$$
where we have extended $f$ to a map $S^{-1}R[X,X^{-1}] \rightarrow (S')^{-1}R'[X,X^{-1}]$ in the obvious way.
\item If $R$ is a field of characteristic zero, then $\gamma_R(\rho, X, \psi)$ coincides with the Deligne-Langlands
gamma factor $\gamma(\rho,X,\psi).$
\end{enumerate}
\end{thm}

Note that if $R$ is reduced and $\ell$-torsion free then the second property characterizes $\gamma_R(\rho,X,\psi)$ uniquely. But any $\rho$ arises by base change from some finite collection of ``universal'' representations $\rho_{\nu}$ over reduced $\ell$-torsion free rings $R_{\nu}$-- see \cite{galois_gamma} for more details. Thus the two properties of Theorem~\ref{thm:local galois gamma} uniquely characterize the association $\rho\mapsto \gamma(\rho,X,\psi)$.

An Artin local $k$-algebra $R$ is an $\ell$-adically complete and separated $W(k)$-algebra. Since a nilpotent lift $\rho\in\mathcal{G}_{R}(n)$ is smooth, it is $\ell$-adically continuous over $R$ in the sense of \cite[\S 2]{galois_gamma}. Therefore Theorem~\ref{thm:local galois gamma} applies to nilpotent lifts.

\section{Characterizing the mod-$\ell$ local Langlands correspondence with nilpotent gamma factors}

\begin{cor}\label{compatibilitywithgammafactors}
Suppose there exists, for every Artin local $k$-algebra $R$, a sequence of surjections
$$L_{R,n}:\mathcal{G}_{R}(n)\to \mathcal{A}_{R}^{\gen}(n),\ \ n\geq 1,$$ satisfying
\begin{enumerate}
\item $L_{R,1}$ is given by local class field theory,
\item For all $\rho$ in $\mathcal{G}_{R}(n)$, for all Artin local $k$-algebras $R'$, for all $\rho'\in\mathcal{G}_{R'}(t)$, for all $t<n$, we have the following equality in $(R\otimes R')[[X]][X^{-1}]$:
\begin{equation*}\label{gammaequalitynil}
\gamma(\rho\otimes\rho',X,\psi) = \gamma(L_{R,n}(\rho)\times L_{R',t}(\rho'),X,\psi).
\end{equation*}
\end{enumerate}
Then the sequence $(L_{R,n})_{n\geq 1}$ is uniquely determined for every $R$.
\end{cor}
\begin{proof}
Fix an Artin local $k$-algebra $R$. Suppose $(\tilde{L}_{R,n})_{n\geq 1}$ is another sequence of maps satisfying (1) and (2). We will show that $\tilde{L}_{R,n}=L_{R,n}$ for all $n$. 

The map $\tilde{L}_{R,1}$ is uniquely determined by condition (1), so that $\tilde{L}_{R,1} = L_{R,1}$.

Now suppose $n\geq 2$. Given $\rho\in \mathcal{G}_{R}(n)$, we have by condition (2) that
$$\gamma(\tilde{L}_{R,n}(\rho)\times L_{R',t}(\rho'),X,\psi) = \gamma(\rho\otimes\rho',X,\psi) = \gamma(L_{R,n}(\rho)\times L_{R',t}(\rho'),X,\psi)$$ for all Artin local $k$-algebras $R'$, for all $\rho'\in\mathcal{G}_{R',t}$, for all $t<n$. Since $L_{R',t}$ is surjective by assumption, $L_{R',t}(\rho')$ runs over all of $\mathcal{A}_{\nil}^{\gen}(t)$ as $\rho'$ and $R'$ vary. By Theorem~\ref{conversefull}, we conclude that $\tilde{L}_{R,n}(\rho)=L_{R,n}(\rho)$.
\end{proof}

\begin{proof}[Proof of Theorem~\ref{characterization}]
This follows immediately from Corollary~\ref{compatibilitywithgammafactors} taking $R=k$, since $\mathcal{G}_k(n) = \mathcal{G}^{ss}_k(n)$. Note that $L_{k,n}^{\gen}$ is thus the restriction to $\mathcal{G}_k^{ss}(n)\subset\mathcal{G}_R(n)$ of any of the maps $L_{R,n}$ appearing in the system of maps of Corollary~\ref{compatibilitywithgammafactors}, for any $R$.
\end{proof}

\bibliography{mybibliography}{}
\bibliographystyle{alpha}
\end{document}